\documentclass[reqno]{amsart}

\usepackage{amsmath}
\usepackage{amsfonts}
\usepackage{hyperref}

\newtheorem{thm}{Theorem}[section]
\newtheorem{lem}{Lemma}[section]

\theoremstyle{definition}
\newtheorem{defin}{Definition}[section]
\everymath{\displaystyle} 

\begin{document}
\title{The Super Patalan Numbers}
\vspace{.5in}

\author[T. M. Richardson]{Thomas M. Richardson \\
Ada, MI  49301}

\email{ribby@umich.edu}

\begin{abstract}
We introduce the super Patalan numbers, a generalization of the super Catalan numbers 
in the sense of Gessel, and prove a number of properties analagous to those of the
super Catalan numbers. The super Patalan numbers generalize the super Catalan numbers similarly to
how the Patalan numbers generalize the Catalan numbers. 
\end{abstract}

\maketitle

\section{Introduction}
We introduce the super Patalan numbers as a gneralization of the super Catalan numbers.
The super Catalan numbers \cite[A068555]{OEIS} were studied by Gessel in his paper on the super ballot numbers \cite{GESSEL}. 
(The term super Catalan numbers is also used to refer to a different sequence, we are 
generalizing the term as used by Gessel.)
Just as the super Catalan numbers form a two dimensional array that extends the Catalan numbers,
the super Patalan numbers of order $p$ form a two dimensional array that extends the Patalan numbers of order $p$.

We start with the definitions of the super Catalan numbers and of the Patalan numbers.

\begin{defin}
Define the \emph{super Catalan numbers} $S(m,n)$ by
\begin{equation*} S(m,n) = \frac{(2m)!(2n)!}{m!n!(m+n)!}.
\end{equation*}
\end{defin}

The Catalan numbers $C_n$ are contained in the super Catalan numbers as
$2C_n = S(n,1) = \frac{2(2n)!}{n!(n+1)!}$.

\begin{defin}
\label{PATALAN_DEF}
Let $p$ be a positive integer with $p>1,$
and let $q$ be a positive integer with $q<p$.
Define the \emph{Patalan numbers of order $p$} to be 
the sequence $a(n)$ with 
$a(0) = 1$, and
\begin{equation}
a(n) = p(pn-1)a(n-1)/(n+1).
\end{equation}
Also define the \emph{$(p,q)$-Patalan numbers} to be
the sequence $b(n)$ with
$b(0) = q$, and
\begin{equation}
b(n) = p(pn-q)b(n-1)/(n+1).
\end{equation}
\end{defin}

The Patalan numbers of order $p$ \cite[A025748, A025749, \ldots, A025757]{OEIS} generalize the Catalan numbers.
In particular the Catalan numbers are the Patalan numbers of order $2$. Also, the Patalan
numbers of order $p$ have generating function $\frac{1-\sqrt{1-p^2x}}{px}$, which generalizes the generating
function of the Catalan numbers.

Now we define the super Patalan numbers as an extension of the Patalan numbers,
and generalizing the super Catalan numbers. 

\begin{defin}
\label{SUPER_PATALAN_DEF}
Define the sequence $Q(i,j)$ of \emph{$(p,q)$-super Patalan numbers} by
\begin{equation}
\label{SPDEF1}
Q(0,0) = 1, 
\end{equation}
\begin{equation}
\label{SPDEF2}
Q(i,0) = p(pi-q)Q(i-1,0)/i,
\end{equation}
and 
\begin{equation}
\label{SPDEF3}
Q(i,j) = p(pj-p+q) Q(i,j-1)/(i+j).
\end{equation}
Let the \emph{super Patalan numbers of order $p$} be the $(p,1)$-super Patalan numbers.
\end{defin}

The super Patalan numbers contain the 
Patalan numbers similarly to how the super Catalan numbers contain the Catalan numbers,
but they do not have quite as simple an
expression as the super Catalan numbers.
In particular, they do not form a symmetric array.
While the super Patalan numbers are not symmetric, they do have a twisted symmetry in that
the arrays of $(p,q)$-super Patalan numbers and $(p,p-q)$-super Patalan numbers are transposes of each other.

The Patalan numbers are contained in the super Patalan numbers just as
the Catalan numbers are contained in the super Catalan numbers.
If $a(n)$ is the sequence of Patalan numbers of order $p$, 
and $P(i,j)$ are the super Patalan numbers of order $p$, 
then the Patalan numbers are contained in the super Patalan numbers as
\begin{equation}
\label{COLUMN_ONE_EQN}
pa(n) = P(n,1).
\end{equation}

It is the author's opinion that to be consistent with equation \eqref{COLUMN_ONE_EQN} 
concerning column $1$ of the super Patalan matrix, 
the Patalan numbers of order $p$
should start $1, \binom{p}{2}$ \cite[A097188]{OEIS}, 
and not start $1,1,\binom{p}{2}$ \cite[A025748]{OEIS}. 
The fact that the Catalan numbers start $1,1$ is explained by 
the Catalan numbers being the Patalan number of order $2$, and $\binom{2}{2} = 1$.

\section{Generating functions}

\begin{thm}
\label{PATALAN_THM1}
The $(p,q)$ super Patalan numbers $Q$ satisfy the identity
\begin{equation}
\label{BINOMIAL_IDENT}
Q(m,n)= (-1)^np^{2(m+n)}\binom{m-q/p}{m+n}.
\end{equation}
\end{thm}

\begin{proof}
Let $R(m,n)= (-1)^np^{2(m+n)}\binom{m-q/p}{m+n}$.
Then $R$ satisfies equations \eqref{SPDEF1}-\eqref{SPDEF3}.
We give some details showing that $R$ satisfies \eqref{SPDEF3}:
\begin{eqnarray}
R(i,j) & = & (-1)^jp^{2(i+j)}\binom{i-q/p}{i+j} \\
& = & (-1)^{j}p^{2(i+j)}\frac{-j+1-q/p}{i+j}\binom{i-q/p}{i+j-1} \\
& = & (-1)^{j-1}p^{2(i+j-1)}\frac{p(pj-p+q)}{i+j}\binom{i-q/p}{i+j-1} \\
& = & \frac{p(pj-p+q)}{i+j}R(i,j-1).
\end{eqnarray}
\end{proof}

Equation \eqref{BINOMIAL_IDENT} generalizes an identity of Gessel \cite[unlabelled equation before equation (31)]{GESSEL}. This indicates that $P(m,n)$ is the coefficient of $x^{m+n}$
in the generating function of 
$(-1)^m(1-p^2x)^{m-q/p}$.

More generally, the above definitions may be extended to define super Patalan numbers for all $m$ and $n$.
\begin{defin}
\label{SUPER_PATALAN_EXTENDED}
Let $m,n$ be integers.
Define the \emph{extended $(p,q)$-super Patalan numbers} $E(m,n)$ to be the coefficient of $x^{m+n}$
in the generating function of
$(-1)^m(1-p^2x)^{m-q/p}$.
\end{defin}

While $E$ is defined in terms of the generating functions of its rows, the twisted symmetry of the super Patalan matrix implies that
$E(m,n)$ is also the coefficient of $x^{m+n}$ in $(-1)^n(1-p^2x)^{n-(p-q)/p}$.

The lower triangular matrix $L$ formed by permuting the columns of $E$ has the interesting property that it
has order $2$ under matrix multiplication.

\begin{thm}
\label{EXTENDED_PATALAN_THM}
Let $L$ be the lower triangular matrix given by $L(m,n) = E(m,-n)$, where $E$ is an extended super Patalan matrix. Then $L^2$ is the identity.
\end{thm}

\begin{proof}
Consider the $(m,n)$ entry of $L^2$ for $n < m$. The product of row $m$ of $L$ and column $n$ of $L$ is
the convolution of row $m$ of $E$ and column $-n$ of $E$. 
The generating function of row $m$ of $E$ is $(-1)^m(1-p^2x)^{m-q/p}$, 
while the generating function of column $-n$ of $E$ is $(-1)^{-n}(1-p^2x)^{-n-(p-q)/p}$. 
Thus the $(m,n)$ entry of $L^2$ is
the coefficient of $x^{m-n}$ in $(-1)^{m-n}(1-p^2x)^{m-n-1}$, which equals $0$.
\end{proof}

Next we consider the two variable generating function of $P$.
\begin{thm}
Let $F(x,y)=\sum P(i,j)x^iy^j$ be the generating function of the super Patalan numbers $P(i,j)$. Then
\begin{equation}
\label{GEN_FCN_TWO_VAR}
F(x,y) = \bigg(\frac{x}{(1-p^2x)^{(p-1)/p}}+\frac{y}{(1-p^2y)^{1/p}}\bigg)\frac{1}{x+y-p^2xy}.
\end{equation}
\end{thm}

\begin{proof}
By Theorem \ref{PATALAN_THM1}, the generating function of the first row of the super Patalan matrix of order $p$ is $g(y)=(1-p^2y)^{-1/p}$ and
the generating function of the first column of the super Patalan matrix of order $p$ is $f(x)=(1-p^2x)^{-(p-1)/p}$.

We will take advantage of the recurrence
\begin{equation}
\label{RECURRENCE_TWO_VAR}
p^2P(i,j) = P(i,j+1) + P(i+1,j).
\end{equation}
Equation \eqref{RECURRENCE_TWO_VAR} implies the equation 
\begin{equation}
\label{RECURRENCE_EQN}
p^2F(x,y) = \frac{F(x,y)-g(y)}{x}+\frac{F(x,y)-f(x)}{y}.
\end{equation}
Solving equation \eqref{RECURRENCE_EQN} for $F(x,y)$ gives equation \eqref{GEN_FCN_TWO_VAR}, as required.
\end{proof}
Equation \eqref{RECURRENCE_TWO_VAR} generalizes an identity attributed to D. Rubenstein by Gessel \cite[equation (36)]{GESSEL}.
Also, equation \eqref{GEN_FCN_TWO_VAR} generalizes a similar expression given by Gessel for the generating function of the super Catalan numbers \cite[equation (37)]{GESSEL}.

\section{Convolutional Recurrence}
The Catalan numbers have a very simple, well known, and interesting convolutional recurrence,
\begin{equation}
\label{CATALAN_RECURRENCE}
C_n = \sum_{k=0}^{n-1} C_kC_{n-k}.
\end{equation}
We show that the Patalan numbers of order $p$ have a similar convolutional recurrence of degree $p$,
and give the explicit recurrence for the Patalan numbers of order 3.

One could derive the recurrences by brute force, exploiting the fact that the generating function for column $1$
of the extended super Patalan numbers is given by the expression $-(1-p^2x)^{1/p}$. 
We will instead work directly with the generating function of the Patalan numbers.
Let $A(x)$ be the generating function of the Patalan numbers of order $p$, 
so that $\displaystyle A(x) = \frac{1-(1-p^2x)^{1/p}}{px}$.
Gessel observed that for $p=3$, $xA(x)$ is the compositional inverse of $x-3x^2+3x^3$ \cite[A097188]{OEIS}.
More generally, $xA(x)$ is the compositional inverse of
$\displaystyle \frac{1-(1-px)^p}{p^2} = -\sum_{k=1}^p \binom{p}{k} p^{k-2}(-x)^k$.
Applying the compositional inverse to the generating function results in a coefficient of zero for the higher degree terms.
Thus we can set the compositional inverse equal to $0$, solve for $x$, and derive a convolutional recurrence from the expression for $x$.
Setting the compositional inverse equal to zero and solving for $x$ gives 
\begin{equation}
\label{COMPOSITIONAL_INVERSE_EQN}
x = \sum_{k=2}^p \binom{p}{k} p^{k-2}(-x)^k.
\end{equation}
Because we are working with the compositional inverse of $xA(x)$, not of $A(x)$, 
we have to be careful when we translate equation \eqref{COMPOSITIONAL_INVERSE_EQN}
to a convolutional recurrence, by subtracting the number of factors of each term from the total degree in the recurrence.
We thus get a recurrence for the Patalan numbers
\begin{equation}
\label{CONVOLUTIONAL_RECURRENCE_EQN}
a(n) = \sum_{k=2}^p \binom{p}{k} p^{k-2}(-1)^k \sum_{i_1+\ldots+i_k = n-k+1} \prod a(i_j)
\end{equation}
It is easily verified that for $p=2$, equation \eqref{CONVOLUTIONAL_RECURRENCE_EQN} reduces to equation \eqref{CATALAN_RECURRENCE}.
For $p=3$, equation \eqref{CONVOLUTIONAL_RECURRENCE_EQN} reduces to 
\begin{equation}
\label{PATALAN3_RECURRENCE}
a(n) = \sum_{k=0}^{n-1} 3a(k)a(n-k-1) - \sum_{i+j+k=n-2} 3a(i)a(j)a(k).
\end{equation}

Equation \eqref{CONVOLUTIONAL_RECURRENCE_EQN} for $n=1$ has only one non-trivial term on the right hand side, 
and it implies that $a(1) = \binom{p}{2}$.

\section{Factorization of the super Patalan matrix}
\begin{defin} 
Define the \emph{reciprocal Pascal matrix} to be the matrix $R$ with $R(i,j) = \binom{i+j}{i}^{-1}$.
\end{defin}
\begin{lem}
\label{FACTORIZATION_LEMMA}
Let $Q$ be the $(p,q)$-super Patalan numbers, and let $G_{p,q}$ be the diagonal matrix with $G_{p,q}(i,i) = Q(i,0)$.
Then 
\begin{equation}
\label{FACTORIZATION_EQN}
Q = G_{p,q} R G_{p,p-q}.
\end{equation}
\end{lem}

The author previously used the factorization of equation \eqref{FACTORIZATION_EQN} to prove that the inverse of the reciprocal Pascal matrix is an integer matrix \cite{RPM}.

Next we prove that the inverse of the Hadamard inverse of the super Patalan matrix is an integer matrix.
\begin{thm}
\label{INVERSE_SUPER_PATALAN_THM}
Let $Q$ be the $(p,q)$ super Patalan matrix, and let $H$ be the $n \times n$ matrix given by $H(i,j) = \frac{1}{Q(i,j)}$ for $0 \le i,j < n$.
Then the inverse of $H$ is an integer matrix.
\end{thm}

\begin{proof}
By Lemma \ref{FACTORIZATION_LEMMA}, 
\begin{equation}
H = G_{p,q}^{-1} B G_{p,p-q}^{-1},
\end{equation}
where $B$ is the Pascal matrix with $B(i,j) = \binom{i+j}{i}$.
Then
\begin{equation}
H^{-1} = G_{p,p-q} B^{-1} G_{p,q}.
\end{equation}
Since $B^{-1}$, $G_{p,q}$, and $G_{p,p-q}$ are all integer matrices, it follows that $H^{-1}$ also is an integer matrix.
\end{proof}

\section{Conclusion}
We have proposed a definition of super Patalan numbers that generalizes the
super Catalan numbers of Gessel, in that the super Catalan numbers are 
the super Patalan numbers of order $p$.
The super Patalan numbers have a number of properties that generalize the
corresponding properties of the super Catalan numbers, in particular
equations \eqref{COLUMN_ONE_EQN}, \eqref{BINOMIAL_IDENT}, \eqref{RECURRENCE_TWO_VAR} and \eqref{GEN_FCN_TWO_VAR}.
We also prove a multiplicative identity for the extended super Patalan matrix, and we give a convolutional recurrence generalizing the well known recurrence for the Catalan numbers.

AMS Classification Numbers: 5A10, 11B37, 11B75, 15A36.

Keywords: Catalan numbers, Patalan numbers, super Catalan numbers, super Patalan numbers, generating function, Pascal matrix.

(Concerned with sequences A025748--A025757, A068555, A097188, A248324--A248326, A248328, A248329, and A248332.)
\end{document}